\documentclass[12pt]{amsart}
\usepackage{a4wide}
\usepackage{amsfonts, amsmath, amsthm, amssymb, stmaryrd, dsfont}
\usepackage{textcomp}
\usepackage{bbm}
\usepackage{enumitem, tabto}
\usepackage{xcolor}
\usepackage{url}
\usepackage{graphicx}
\usepackage{fancybox}
\usepackage{awesomebox}
\usepackage{tikz}
\usepackage{pgfplots}
\usepackage[Glenn]{fncychap}
\usepackage{fancyhdr}
\usepackage{lastpage}
\usepackage{import}

\usepackage{diagbox}
\usepackage{multirow} 
\usepackage{tabularx}

\usepackage{etoolbox}

\usepackage{hyperref}

\usepackage{blindtext}
\usepackage{multicol}
\usepackage{wrapfig}

\usepackage{pdfpages}

\usepackage{comment}

\newcommand{\N}{\mathbb{N}}

\newcommand{\R}{\mathbb{R}}
\newcommand{\C}{\mathbb{C}}
\newcommand{\D}{\mathbb{D}}

\newcommand{\T}{\mathbb{T}}
\newcommand{\Lcont}{\mathcal{L}}

\newcommand{\norm}[1]{\left\Vert #1\right\Vert}
\newcommand{\abs}[1]{\left\lvert #1 \right\rvert}
\newcommand{\scal}[2]{\ensuremath{\left\langle #1|#2 \right\rangle}\xspace}

\newcommand{\codim}{\text{codim}}

\usepackage{indentfirst}
\makeatletter
\@addtoreset{equation}{section}
\makeatother
\theoremstyle{plain}
\newtheorem{statement}{}[section]
\newtheorem{theo}[statement]{Theorem}
\newtheorem{prop}[statement]{Proposition}

\newtheorem{coro}[statement]{Corollary}
\newtheorem{lem}[statement]{Lemma}

\newtheorem{quest}[statement]{Question}

\theoremstyle{definition}

\newtheorem{rem}[statement]{Remark}

\title[Embedding of some classes of operators]{Embedding of some classes of operators into strongly continuous semigroups}
\author{Isabelle Chalendar and Romain Lebreton}
\address{Laboratoire d'Analyse et de Mathématiques Appliqués, Université Gustave Eiffel, 77454 Champs-sur-Marne Cédex 2}
\email{isabelle.chalendar@univ-eiffel.fr}
\address{Laboratoire Paul Painlevé, Université de Lille, 59655 Villeneuve d'Ascq Cédex}
\email{romain.lebreton@univ-lille.fr}

\thanks{The second author is supported by Labex CEMPI (ANR-11-LABX-0007-01)}

\begin{document}

\maketitle

\begin{abstract}
    In this paper we study the embedding problem of an operator into a strongly continous semigroup. We obtain characterizations for some classes of operators, namely composition operators  and analytic Toeplitz operators on the Hardy space $H^2$. In particular, we focus  on the  isometric ones using  the necessary and sufficient condition observed by T. Eisner.
\end{abstract}

\section{Introduction}

If $T$ is a bounded linear operator on a Banach space $E$, then $T$ is said to be \textit{embeddable} into a strongly continuous semigroup if there exists $(T_t)_{t \ge 0}$ a family of bounded linear operators on $E$ which is  a strongly continous semigroup such that $T = T_1$.

This notion has been studied by T. Eisner especially in the monograph \cite{TanEis} and it turned out to be a very interesting and challenging problem. There is no known necessary and sufficient condition for any operator $T$ on a Banach space, but T. Eisner showed  that if $T$ is embeddable, then $\dim(\ker(T))$ and $\codim(\overline{\mbox{Im}}(T))$ are either $0$ or $\infty$. This condition shows us  that the forward shift operator $S$ on $H^2$ and all its powers are not embeddable into a strongly continous semigroup on $H^2$.

One of the first easy examples is the embedding of the Volterra operator $V : f \in L^p([0,1]) \longmapsto Vf(x) = \int_0^x f(s) ~ds$ into the Riemann-Liouville semigroup on $L^p([0,1])$, for $1 \le p < \infty$ \cite{alam2024,WolfgangVolterra}.

However, there exist some necessary and sufficient conditions for special classes of operators and we higlight here the following result on isometric operators obtained by T. Eisner in   \cite{TanEis1} and \cite[Theorem V.1.19]{TanEis}.  

\begin{theo}\label{PlongIsoTanja}
    Let $V : H \rightarrow H$ be an isometry on a Hilbert space $H$. Then $V$ is embeddable into a $C_0$-semigroup on $H$ if and only if $V$ is unitary or $\codim(VH) = \infty$.
\end{theo}

This theorem is very useful in the case of composition operators or analytic Toeplitz operators where it is easy to characterize the  isometric ones. The main  goal of this paper is then  to describe the embedding of such isometric operators, and to make the associated  semigroup  explicit.

The paper is organized as follows. In section \ref{background}, we recall some useful properties on (analytic) semigroup theory of operators and the main tools required on Hardy spaces. We also give a key lemma on Blaschke products  for the main result on composition operators. Section \ref{embedding} is concerned with  the embedding of isometric composition operators and also the embedding  of analytic Toeplitz operators. Finally, in Section \ref{PropOpIso}, we show that the embedding of an isometry into a semigroup of contractions $(T_t)_{t\geq 0}$ implies that  each $T_t$ is  isometric. Moreover, the  embedding of an isometry into a semigroup  $(T_t)_{t\geq 0}$ (not necessarily contractive) implies that, for each $t>0$,  $T_t$ is never compact. 

\section{Background and preliminaries}\label{background}

\subsection{Strongly continuous semigroups of operators}

Let $E$ be a Banach space and $(T_t)_{t \ge 0} \subset \Lcont(E)$, the space of all linear and bounded operators on $E$ endowed with its usual norm. We say that $(T_t)_{t \ge 0}$ is a \textit{strongly continuous semigroup} or just a \textit{$C_0$-semigroup} if $$T_0 = Id, \quad T_{t+s} = T_t \circ T_s ,\quad  t,s > 0$$ and for all $x \in E$, $$t \in \R_+ \longmapsto T_t x \text{~is continuous i.e.~} \norm{T_t x - x}_E \underset{t \rightarrow 0^+}{\longrightarrow} 0.$$
\color{black} See for example \cite{SemiGroup} for an introduction to semigroup theory of operators.

We recall here that an operator $T \in \Lcont(E)$ is \textit{embeddable} into a $C_0$-semigroup on $E$ if there exists $(T_t)_{t \ge 0}$ a $C_0$-semigroup on $E$ such that $T = T_1$. In this case, we write $$T \hookrightarrow (T_t)_{t \ge 0}.$$ 

\subsection{Analytic semiflows on \texorpdfstring{$\D$}{D}}\label{AnalyticSF}

Let $\D = \{z \in \C : \abs{z} < 1 \}$ be the open unit disc of the complex plane $\C$ and $(\varphi_t)_{t \ge 0}$ be a family of analytic self-maps of $\D$. We say that $(\varphi_t)_{t \ge 0}$ is a \textit{semiflow of analytic self-maps of $\D$} if $$\varphi_0 = Id, \quad \varphi_{t + s} = \varphi_t \circ \varphi_s ,\quad  t,s \geq 0$$ and for all $z\in\D$, $$t \in \R_+ \longmapsto \varphi_t (z)\text{~is continuous}.$$
Note that the pointwise continuity assumption is equivalent to the uniform continuity on all compact subsets of $\D$ via Montel's theorem.


It is a well known fact that when $(\varphi_t)_{t \ge 0}$ is a semiflow of analytic self-maps of $\D$, each function $\varphi_t$ is one-to-one. See \cite{BraConDia} for a proof using Cauchy-Lipschitz's theory or \cite{DirChalBen} for an alternative elementary proof. We recommend  \cite{BraConDia} for a very complete state of the art  of  analytic semiflow theory.

In the same way, we say that $\varphi \in \mbox{Hol}(\D)$ is \textit{embeddable} into a semiflow of analytic self-maps of $\D$ if there exists $(\varphi_t)_{t \ge 0}$ a semiflow of analytic self-maps of $\D$ such that $\varphi = \varphi_1$. In this case, we write $$\varphi \hookrightarrow (\varphi_t)_{t \ge 0}.$$ As an example, let $\varphi$ be  an elliptic automorphism of $\D$ i.e.  $\varphi$ is an holomorphic and bijective function on $\D$ with a fixed point $\alpha \in \D$. Equivalently, there exist $\theta\in \R$ and  $\alpha\in\D$ such that 
\begin{equation}\label{AutElliptic}
    \varphi = \tau_\alpha \circ R_\theta \circ \tau_\alpha,\mbox{ where } R_\theta : z \longmapsto e^{i\theta}z\mbox{ and } \tau_\alpha : z \longmapsto \frac{\alpha - z}{1 - \overline{\alpha}z}.
\end{equation}
Note that $\tau_\alpha$ is an automorphism of $\D$ which coincides with  its inverse. Thus, $\varphi$ is embeddable into the following semiflow of analytic self-maps of $\D$: $$  \varphi_t = \tau_\alpha \circ R_{\theta t} \circ \tau_\alpha, \mbox{ with } R_{\theta t} : z \longmapsto e^{it\theta}z,\quad t\geq 0$$ 

\subsection{Hardy spaces and Blaschke products}

Let $\T = \{z \in \C : \abs{z} = 1 \}$ be the unit circle endowed with the Lebesgue measure $m$. For $0<p<\infty$, we consider the Hardy space $H^p=H^p(\D)$ which consists of functions $f$ holomorphic on $\D$ satisfying
\[
\norm{f}_p= \sup_{0 \le r < 1}\left(\int_{\T} \abs{f(r\zeta)}^p\,dm(\zeta)\right)^{1/p}<\infty.
\]
Recall that for $p=2$, $H^2$ is a reproducing kernel Hilbert space whose kernel is given by 
\[
k_\lambda(z) = \frac{1}{1 - \overline{\lambda}z}, \quad \lambda, z \in \D
\] meaning that $$f(\lambda) = \scal{f}{k_\lambda}_2, \quad f \in H^2, \quad \lambda \in \D.$$
Moreover, we have $\mbox{Span}_{H^2}(k_\lambda : \lambda \in \D) = H^2$, where $\mbox{Span}_{\mathcal H}(A)$ denotes the closure in $\mathcal H$ of the subspace consisting of finite linear combinations of elements of $A$, where $A$ is a family of vectors in a Hilbert space $\mathcal H$. Define by $H_0^2$ the space of functions $f \in H^2$ such that $f(0) = 0$.

We also define $H^\infty = H^\infty(\D)$ to be the class of bounded analytic functions on $\D$, endowed with the sup norm defined by $\norm{f}_\infty = \underset{z \in \D}{\sup} \abs{f(z)}$.

We also recall that a Blaschke product is a function of the form 
\begin{equation}\label{ProdBla}
    B(z) = e^{i\beta}z^k\prod_{n \ge 1} \frac{\abs{\alpha_n}}{\alpha_n}\frac{\alpha_n - z}{1 - \overline{\alpha_n}z}, \quad z \in \D
\end{equation} where $\beta \in \R$, $k \in \N\cup\{0\}$ and $(\alpha_n)_{n \ge 1}$ is a finite or infinite sequence of $\D \backslash \{0\}$ satisfying $\sum_{n \ge 1} (1-\abs{\alpha_n}) < \infty$. Then $B$ is an inner function on $\D$.  When  $(\alpha_n)_{n \ge 1}$ is a finite   sequence of $\D$, we say that $B$ is a finite Blaschke product and one can easily check that such $B$ are continuous on the closed unit disc. We refer the reader to \cite{BlaFin} for more details about finite Blaschke products.


In the sequel we consider the finite Blaschke product associated with a finite sequence $(\alpha_n)_{1 \le n \le N}\subset \D$ defined  by 
\begin{equation}\label{BlaschkeFini}
    B(z) = \prod_{i=1}^N \frac{\alpha_i - z}{1 - \overline{\alpha_i}z}, \quad z \in \D.    
\end{equation}
We know that the equation $B(z) = \beta$ for $\beta \in \D$ has exactly $N$ solutions in $\D$, taking into account the   multiplicity  (see \cite{BlaFin}). The following lemma is the key to differentiate them.

\begin{lem}\label{lemzerodistinctEqB}
    Let $\beta \in \D \backslash B( \mbox{Zero}(B'))$. Then the equation $B(z) = \beta$ has exactly $N$ distinct solutions in $\D$. 
\end{lem}
\begin{proof}
    Let $\beta \in \D$. Then $B(z) = \beta$ is equivalent to a polynomial  equation  of the form
    \[ P(z)-\beta Q(z)=0\mbox{ with }P(z)= \prod_{i=1}^N (\alpha_i - z)  \mbox{ and }Q(z)=\prod_{i=1}^N (1 - \overline{\alpha_i}z).\]
    Since $(-1)^{N}\left(1 - \beta \prod_{i=1}^N \overline{\alpha_i}\right) \neq 0$ and since $B$ maps $\D$ to $\D$, $\T$ to $\T$ and 
    $\{z\in\C:|z|>1\}$ to itself, there are N solutions in $\D$. It remains to prove that for suitable $\beta$ the solutions are distinct. 
    
    Note that $$B'(z) = \frac{P'(z)Q(z) - P(z)Q'(z)}{Q(z)^2}$$ and thus 
    $$P(z) - \beta Q(z) = 0 \text{~and~} P'(z) - \beta Q'(z) = 0\Longrightarrow B'(z)=0.$$ 
   It follows that  for $\beta \in \D \backslash B( \mbox{Zero}(B'))$, the equation $B(z) = \beta$ has exactly $N$ distinct solutions in $\D$.
\end{proof}

We also recall the well known and very useful Frostman's theorem \cite{espMod} whose assertion is the following. Let $\theta$ be an arbitrary inner function.  Then there exists a set of measure zero with respect to the area measure (even a set of capacity  zero) $\Omega \subset \D$ such that for every $\lambda \in \D \backslash \Omega$, the so-called \textit{Frostman's transform} $$\theta_\lambda := \tau_\lambda \circ \theta = \frac{\lambda - \theta}{1 - \overline{\lambda}\theta}$$ is a Blaschke product with simple zeros.  Here recall that $\tau_\lambda$ is the usual automorphism of $\D$ defined by $\tau_\lambda(z) = \frac{\lambda - z}{1 - \overline{\lambda}z}$ for $z \in \D$.

\section{Embedding results}\label{embedding}

\subsection{Composition operators on \texorpdfstring{$H^2$}{H2}}

First of all, recall that the main goal of this paper is to describe classes of operators which can be embedded into a semigroup of operators. The isometric operators are of particular interest  due to  the necessary and sufficient condition in Theorem~ \ref{PlongIsoTanja} on Hilbert spaces. The choice of composition operators on $H^2$ is relevant since  the isometric ones as well as the ones which are similar to isometries  are fully characterized and, moreover, this class is quite rich.  
We define the \textit{composition operator} $C_\varphi$ with symbol $\varphi$ on $H^2$ by $$C_\varphi : f \longmapsto f \circ \varphi, \quad f \in H^2.$$ This operator is well-defined and bounded on $H^2$. See \cite{CompOpShap} for further information about those operators on $H^2$.

Moreover, $C_\varphi$ is an isometry on $H^2$ if and only if $\varphi$ is inner and $\varphi(0) = 0$. In \cite{SimIsoFB}, F. Bayart showed that $C_\varphi$ is similar to an isometry on $H^2$ if and only if $\varphi$ is inner and there exists $\alpha \in \D$ such that $\varphi(\alpha) = \alpha$. In other words, $C_\varphi$ is similar to an isometry on $H^2$ if and only if its symbol $\varphi$ is an elliptic inner function.

 In \cite{AsymBehCompOp}, W. Arendt, I. Chalendar, M. Kumar and S. Srivastava  showed the following results.
\begin{itemize}[label = \textbullet]
    \item On the Bergman space $\mathcal{A}^2$ (and even on its weighted versions $\mathcal{A}^2_\beta$ for $\beta > -1)$, $C_\varphi$ is similar to an isometry if and only if $\varphi$ is an elliptic automorphism of $\D$. Thus, we will find a natural embedding according to \eqref{AutElliptic} for $C_\varphi$, described later by Remark \ref{SGSF}.
    \item On the classical Dirichlet space $\mathcal{D}$, $C_\varphi$ is similar to an isometry if and only if $\varphi$ is a univalent full map with a fixed point in $\D$ and the counting function $n_\varphi$ associated to $\varphi$ is essentially radial (see \cite[Section 6]{AsymBehCompOp}). Note that the existing criteria for the boundedness of $C_\varphi$ is not that explicit  and so the similarity to an isometry is much less easy to handle. 
\end{itemize}

From now on, the space on which we study our operators are  defined on the Hardy space $H^2$. Let $\varphi : \D \to \D$ be analytic. 
\begin{rem}\label{SGSF}
    Observe that if $\varphi \hookrightarrow (\varphi_t)_{t \ge 0}$ where $(\varphi_t)_{t \ge 0}$ is a semiflow of analytic self-maps of $\D$, then $C_\varphi \hookrightarrow (C_{\varphi_t})_{t \ge 0}$ where $ (C_{\varphi_t})_{t \ge 0}$ is a $C_0$-semigroup on $H^2$. 

Conversely if $C_\varphi \hookrightarrow (T_t)_{t \ge 0}$ where $(T_t)_{t \ge 0}$ is a $C_0$-semigroup of composition operators on $H^2$, then applying $T_t$ to $e_1(z) := z$ and using the fact that the convergence in $H^2$ implies the pointwise convergence, we get the existence of $(\varphi_t)_{t \ge 0}$ a semiflow of analytic self-maps of $\D$ such that $T_t = C_{\varphi_t}$.  
\end{rem}


With the reproducing kernel Hilbert space property of $H^2$, we can give the following first sufficient condition of embedding for isometric composition operators. This is a key to understand the strategy of the proof of the main result. 

\begin{lem}\label{exi2preim}
    Let $\varphi : \D \rightarrow \D$ be an inner function such that $\varphi(0) = 0$. Assume that there exists a sequence $(z_k)_{k\geq 0}$ of distinct points   in $\D $ such that each $z_k$  has at least two preimages by $\varphi$. Then $C_\varphi$ is embeddable into a $C_0$-semigroup on $H^2$.
\end{lem}

\begin{proof}
Since $\varphi$ is an inner function such that $\varphi(0) = 0$, it comes that $C_\varphi$ is an isometry on $H^2$. Moreover, by hypothesis, for $K \in \N$, there exist $(w_1, ..., w_K)$ and $(w_1', ..., w_K') \in \D^K$ such that $$\forall 1 \le i \le K, \quad w_i \neq w_i' \text{~and~} \varphi(w_i) = \varphi(w_i') = z_i.$$ For each $1 \le i \le K$, define the function $f_i \in H^2$ by $f_i = k_{w_i} - k_{w_i'}$. Then, for all $f \in H^2$ and for every $1 \le i \le K$, $$\scal{C_\varphi f}{f_i}_2 = \scal{f \circ \varphi}{k_{w_i}}_2 - \scal{f \circ \varphi}{k_{w_i'}}_2 = f \circ \varphi(w_i) - f \circ \varphi(w_i') = 0.$$ Theorefore, we obtain that $(f_i)_{1 \le i \le K} \subset \mbox{Im}(C_\varphi)^\perp$. Since $(f_i)_{1 \le i \le K}$ is  a set of $K$ linearly independent functions, where $K$ is arbitrary large, we get that \begin{equation*}
    \codim(\mbox{Im}(C_\varphi)) = \dim(\mbox{Im}(C_\varphi)^\perp) = \infty. 
\end{equation*}
Finally, we conclude with Theorem \ref{PlongIsoTanja} and then $C_\varphi$ is embeddable into a $C_0$-semigroup on $H^2$.
\end{proof}

Recall that if $\varphi$ is an elliptic automorphism, as in Subsection \ref{AnalyticSF}, then $C_\varphi$ is embeddable into a semigroup of composition operators on $H^2$ (see \eqref{AutElliptic} and Remark \ref{SGSF}). We refer the reader to \cite{BraConDia} for the remaining  automorphism cases, for which there exist natural embeddings thanks to Remark \ref{SGSF}.

\begin{theo}
    Every composition operator  $C_\varphi$ which is similar to an isometry on $H^2$ is embeddable into a $C_0$-semigroup $(T_t)_{t\geq 0}$ on $H^2$, which is not a semigroup of composition operators, unless $\varphi$ is an automorphism.
\end{theo}

\begin{proof}
	 Let $\varphi : \D \to \D$ be an analytic function. Then $C_\varphi$ is similar to an isometry on $H^2$ if and only if $\varphi$ is inner and there exists $\alpha \in \D$ such that $\varphi(\alpha) = \alpha$.  
    The case when $\varphi$ is an elliptic automorphism is already done thanks to Remark \ref{SGSF}. From now on, assume that $\varphi$ is a  non automorphic inner function  with a fixed point $\alpha\in \D $.  
    
    Let us first remark  that $\varphi$ is not injective. Indeed, 
        \begin{itemize}[label = \textbullet]
            \item if $\varphi$ is a Blaschke product as \eqref{ProdBla}, then $\varphi$ is not injective since $0$ has at least two preimages; 
            \item if $\varphi$ is not a Blaschke product, then according to  Frostman's theorem, the map $\tau_a \circ \varphi = \frac{a - \varphi}{1 - \overline{a}\varphi}$ is a Blaschke product $B$ with simple zeros for almost all $a \in \D$. Therefore, $\varphi = \tau_a^{-1}(B)$ is not injective.
        \end{itemize}
       In that case, there is no semiflow of analytic self-maps of $\D$ in which $\varphi$ is embeddable. Therefore, from Remark \ref{SGSF}, $C_\varphi$ is not embeddable into a $C_0$-semigroup of composition operators on $H^2$.
    For $\varphi(\alpha) = \alpha$, let us consider the application defined as follows 
    \[\psi := \tau_\alpha \circ \varphi \circ \tau_\alpha.\]
     Then $\psi$ is an inner function such that $\psi(0) = 0$. Thus $C_\psi$ is an isometry of $H^2$ and we get $C_\varphi = C_{\tau_\alpha} \circ C_\psi \circ C_{\tau_\alpha}$. Since $C_{\tau_\alpha}$ is an isomorphism of $H^2$, it remains to show that $C_\psi$ is embeddable into a $C_0$-semigroup on $H^2$.  
 \begin{itemize}[label=\textbullet]
 	\item If $\psi$ is a finite Blaschke product $B$, the conclusion follows  from Lemma \ref{lemzerodistinctEqB} and Lemma~\ref{exi2preim}, taking $(z_k)_{k\geq 0}$ in $\D\backslash B(\mbox{Zero}(B'))$. 
 	\item If $\psi$ is inner but not a finite Blaschke product,  according to Frostman's theorem, the map $\tau_\gamma \circ \psi =: B$ is a Blaschke product with simple zeros for almost all $\gamma \in \D$.  In that case, we have 
 	\begin{align}\label{ImCpsiCB}
 		C_\psi H^2 & = \{ f \circ \psi : f \in H^2 \} \\ \nonumber & = \{ (f \circ \tau_\gamma) \circ (\tau_\gamma \circ \psi) : f \in H^2 \} \\ \nonumber & = \{ g \circ B : g \in H^2 \} = C_B H^2.
 	\end{align}
 	Denote by $(w_k)_{k \ge 1}$ the sequence of simple zeros of $B$. Considering the  sequence  $(k_{w_i} - k_{w_j})_{i,j \ge 1, i \neq j}$ of linearly independent functions of $H^2$, we deduce from \eqref{ImCpsiCB} that 
    $$\codim(\mbox{Im}(C_\psi)) = \dim(\mbox{Im}(C_B)^\perp) = \infty.$$  
The embedding of $C_\psi$ follows  from Theorem \ref{PlongIsoTanja}. 
 \end{itemize}    
      Finally, we conclude that $C_\varphi$ is embeddable into the $C_0$-semigroup $(C_{\tau_\alpha}T_tC_{\tau_\alpha})_{t \ge 0}$ on $H^2$, where $C_\psi$ is embedded into a  $C_0$-semigroup denoted by $(T_t)_{t \ge 0}$ on $H^2$.
 \end{proof}
In order to describe the semigroup in which $C_\varphi$ is embeddable, we need the following lemma which appears in \cite[Lemma 5]{NORDGREN}. For the sake of completeness,  we include a slightly different proof.
\begin{lem}\label{lem:inter}
Let $\psi$ be an inner function such that $\psi(0)=0$ and such that $\psi$ is not an automorphism. Then $\bigcap_{n\geq 0}C_\psi^n H^2_0=\{0\}$ and  thus  $\bigcap_{n\geq 0}C_\psi^n H^2=\C \mathds 1$, where $\mathds 1$ stands for the constant function equal to $1$. 
\end{lem}
\begin{proof}
	Let $g \in \bigcap_{n\geq 0}C_\psi^n H_0^2$. Then for each $n \ge 1$, there exists $f_n \in H_0^2$ such that $g(z) = f_n(\psi^{[n]}(z))$ for every $z \in \D$ with $\psi^{[n]} = \psi \circ \cdots \circ \psi$ ($n$ times). Morever $\norm{g}_2 = \norm{f_n}_2$ since $C_\psi$ is isometric. Note that if $g \neq 0$, then there exists  $z_0 \in \D$, $z_0 \neq 0$ such that $\abs{g(z_0)} > 0$. 
	Since $f_n \in H_0^2$, $f_n(0) = 0$, there exists $g_n \in H^2$ such that $f_n(z) = zg_n(z)$ for every $z \in \D$, with $\norm{g_n}_2 = \norm{g}_2$. Finally, we get that 
	\begin{align*}
		\abs{f_n(\psi^{[n]}(z_0))} & = \abs{\psi^{[n]}(z_0)}\abs{g_n(\psi^{[n]}(z_0))} \\ & = \abs{\psi^{[n]}(z_0)}\abs{\scal{g_n}{k_{\psi^{[n]}(z_0)}}_2} \\ & \le \abs{\psi^{[n]}(z_0)}\norm{g}_2 \frac{1}{\sqrt{1 - \abs{\psi^{[n]}(z_0)}^2}}.
	\end{align*}
Since $0$ is the Denjoy-Wolff point of $\psi$, we get  $\psi^{[n]}(z_0) \underset{n \rightarrow \infty}{\longrightarrow} 0$.
Then $\abs{f_n(\psi^{[n]}(z_0))} \underset{n \rightarrow \infty}{\longrightarrow} 0$ and we get $g(z_0) = 0$, a contradiction. Consequently, $\bigcap_{n\geq 0}C_\psi^n H^2_0=\{0\}$.
  
The second assertion of the lemma follows from the fact that if $f \in \bigcap_{n \ge 0} C_\psi^n H^2$ then $f - f(0)\mathds{1} \in \bigcap_{n \ge 0} C_\psi^n H_0^2$. Indeed, in that case, for each $n \ge 1$, there exists $h_n \in H^2$ such that $f = C_\psi^n h_n$. Thus $f - f(0)\mathds{1} = C_\psi^n(h_n - f(0)\mathds{1})$. Since $f(0) = h_n(\psi^{[n]}(0)) = h_n(0)$, we have $f - f(0)\mathds{1} = C_\psi^n(h_n - h_n(0)\mathds{1})$ with $h_n - h_n(0)\mathds{1} \in H_0^2$. Therefore, $f - f(0)\mathds{1} \in  \bigcap_{n \ge 0} C_\psi^n H_0^2$. 
\end{proof}

\begin{coro}\label{cor:form}
Let $\varphi$ be an inner function  with a fixed  point $\alpha\in\D$ and such that $\varphi$ is not an automorphism. Then $C_\varphi$ on $H^2$ is  embeddable into the $C_0$-semigroup \[\left( M_{e^{it\theta}} \oplus C_{\tau_\alpha}U^*S_tU C_{\tau_\alpha}\right)_{t \ge 0},\]
with respect to the decomposition  $\C\mathds 1 \oplus H_0^2$ 
 where $U : H_0^2 \rightarrow L^2(\R_+, \mbox{Im}(C_\psi)^\perp)$ is unitary, $\theta \in \R$ and $(S_t)_{t \ge 0}$ is the right semigroup on $L^2(\R_+, \mbox{Im}(C_\psi)^\perp)$ defined by \[(S_t f)(s)=\begin{cases}
 	f(s-t) & s-t\geq 0\\
 	0 & s-t<0 
 \end{cases}. \]
\end{coro}	
\begin{proof}
First note that 
\begin{equation}\label{eq:sim}
	C_\varphi = C_{\tau_\alpha}C_\psi C_{\tau_\alpha}
\end{equation} where $\psi$ is an inner function such that $\psi(0)=0$ and $\psi$ is not an automorphism.  Therefore $C_\psi$ is isometric. 
Using Wold's decomposition and Lemma \ref{lem:inter}, we obtain that $H^2 = F \oplus^\perp G$ with $F := \bigcap_{n \ge 0} C_\psi^n H^2 = \C \mathds{1}$ and, by the properties of the orthogonal direct sum, $G := \bigoplus_{n \ge 0} C_\psi^n(H^2 \ominus C_\psi H^2) = H_0^2$. In that case  $(C_\psi)_{|F}$ is unitary on a vector space of dimension $1$. Therefore $(C_\psi)_{|F} = M_{e^{i\theta}}$ for some  $\theta\in\R$. On the other side, $(C_\psi)_{|G}$ is unitarily equivalent to the right shift $S$ on $\ell^2(\N, \mbox{Im}(C_\psi)^\perp)$. According to the embedding of $S$ by \cite[Proposition V.1.18]{TanEis}, we obtain the embedding of $C_\psi$ into the $C_0$-semigroup $\left (M_{e^{it\theta}} \oplus U^*S_tU \right)_{t \ge 0}$ where $U : H_0^2 \rightarrow L^2(\R_+, \mbox{Im}(C_\psi)^\perp)$ is unitary, $\theta \in \R$ and $(S_t)_{t \ge 0}$ is the right semigroup on $L^2(\R_+, \mbox{Im}(C_\psi)^\perp)$. We conclude the proof using \eqref{eq:sim}.
\end{proof}

\begin{rem}
    Another special case is when the symbol $\varphi$ is a \textit{linear fractional map} of the unit disc. Indeed, we have a complete characterization of the embedding of $\varphi$ according to its fixed points by \cite[Proposition 3.4]{LFMPlong}. Then, we have the natural embedding of $C_\varphi$ by Remark \ref{SGSF}. However, there are some examples where the embedding of $C_\varphi$ into a $C_0$-semigroup of composition operators on $H^2$ is not possible. To that aim, it suffices  to consider $\varphi$ the attractive elliptic function on $\D$ defined by $\varphi(z) = \frac{z}{z-2}$. Indeed, $\varphi$ does not satisfy the following required condition:
    \begin{equation}\label{condPlongEllAttrac}
            \abs{\overline{\alpha} - \frac{1}{\beta}}l \le \abs{\varphi'(\alpha)} \abs{1 - \frac{\alpha}{\beta}},
    \end{equation} where $\alpha \in \D$ is its Denjoy-Wolff point, $\beta \in (\C \cup \{\infty\}) \backslash \D$ its repulsive fixed point and $l = l(\varphi'(\alpha))$ the lenght of the canonical spiral associated with $\varphi'(\alpha) \in \D \backslash \{0\}$. Consequently, $\varphi$ is not embeddable into a semiflow of analytic self-maps of $\D$. We refer the reader to \cite{BraConDia} for more information about Denjoy-Wolff theory. 
\end{rem}

Let us now introduce weighted composition operators. Let $w \in H^2$ and $\varphi : \D \to \D$ be analytic. We define the \textit{weighted composition operator} $C_{w, \varphi}$ with symbol $\varphi$ and weight $w$ by $$C_{w, \varphi} : f \longmapsto w(f \circ \varphi), \quad f \in H^2.$$ In \cite{OpCompPondIso}, R. Kumar and J. Partington proved that $C_{w, \varphi}$ is an isometry on $H^2$ if and only if $\varphi$ is inner, $\norm{w}_2 = 1$ and $\scal{w}{w \varphi^n}_2 = 0$ for every $n \ge 1$.

In \cite{OpCompPondIso2},  I. Chalendar and J. Partington obtained the following result: if $\varphi$ is inner, then there exists a weight $w \in H^2$ such that $C_{w, \varphi}$ is an isometry on $H^2$.

The combination of these two results gives the following sufficient condition about the embedding of weighted composition operators. The main interest of this assertion is that, provided that we make an appropriate choice of the weight, it is not required that the symbol of the composition operator  has a fixed point in $\D$.  

\begin{theo}
    Let $\varphi$ be an inner function. Then there exists a weight $w \in H^2$ such that $C_{w, \varphi}$ is embeddable into a $C_0$-semigroup on $H^2$.
\end{theo}

\begin{proof}
    Since $\varphi$ is inner, there exists a weight $w \in H^2$ such that $C_{w, \varphi}$ is an isometry on $H^2$. Moverer, $w$ satisfied $\norm{w}_2 = 1$ and $\scal{w}{w \varphi^n}_2 = 0$ for every $n \ge 1$. It remains to show that $\codim (\mbox{Im}(C_{w, \varphi})) = \infty$. Note that, for every $\lambda \in \D$ and $f \in H^2$, we have $$\scal{C_{w, \varphi}f}{k_\lambda}_2 = w(\lambda)C_\varphi f(\lambda) = w(\lambda) (f \circ \varphi(\lambda)).$$ Then, we deduce that $\scal{C_{w, \varphi}f}{k_\lambda}_2 = 0$ if and only if $w(\lambda) = 0$ or $f \circ \varphi(\lambda) = 0$. Take $w = Bm$ where  $B$ is an infinite Blaschke product associated with a sequence $(\lambda_n)_{n \ge 1} \subset \D$ satisfying $\sum_{n \ge 1} (1 - \abs{\lambda_n}) < \infty$ and where $m \in H^2$ satisfies $\norm{m}_2 = 1$ and $\scal{m}{m \varphi^n}_2 = 0$ for every $n \ge 1$. Notice that since $B$ is inner, we get $\norm{w}_2 = 1$ and $\scal{w}{w \varphi^n}_2 = 0$ for every $n \ge 1$. Thus $C_{w, \varphi}$ is an isometry on $H^2$ such that $$\scal{C_{w, \varphi}f}{k_\lambda}_2 = B(\lambda) m(\lambda) C_\varphi f(\lambda) = 0, \quad \lambda \in \{\lambda_n :n \ge 1\} \subset \D.$$ In other words, we deduce that $\mbox{Span}_{H^2}(k_\lambda : \lambda \in \mbox{Zero}(B)) \subset \mbox{Im}(C_{w, \varphi})^\perp$ and \[\codim (\mbox{Im}(C_{w, \varphi})) = \dim (\mbox{Im}(C_{w, \varphi})^\perp) = \infty.\] Finally, for such a $w \in H^2$, $C_{w, \varphi}$ is embeddable into a $C_0$-semigroup on $H^2$ by Theorem \ref{PlongIsoTanja}. 
\end{proof}

\begin{rem}
    The form of the $C_0$-semigroup in which $C_{w, \varphi}$ is embeddable is less explicit than the one gived in  Corollary \ref{cor:form}.   Indeed, for $\varphi$ an inner function and $w \in H^2$ the weight such that $C_{w, \varphi}$ is an isometry on $H^2$, by the Wold's decomposition, $H^2 = F \oplus^\perp G$ where $(C_{w, \varphi})_{|F}$ is unitary and  $(C_{w, \varphi})_{|G}$ is unitarily equivalent to the right shift on $\ell^2(\N, \mbox{Im}(C_{w, \varphi})^\perp)$. Then, by \cite[Theorem V.1.14]{TanEis}, $C_{w, \varphi}$ is embeddable into the $C_0$-semigroup $$\left (Z^* \left (e^{t Log (m)} \right)Z \oplus U^*S_tU \right )_{t \ge 0}$$ where $\mu$ is a Borel measure, $m \in L^\infty(\sigma((C_{w, \varphi})_{|F}), \mu)$ is measurable and \[Z : F \rightarrow L^2(\sigma((C_{w, \varphi})_{|F}), \mu),\,\,\, U : G \rightarrow L^2(\R_+, \mbox{Im}(C_{w, \varphi})^\perp)\]
     are unitary operators. 
\end{rem}

\subsection{Analytic Toeplitz operators on \texorpdfstring{$H^2$}{H2}}

Let $\varphi \in L^\infty(\T)$. We define the \textit{Toeplitz operator} $T_\varphi$ with symbol $\varphi$ by $$T_\varphi : f \longmapsto P_+(\varphi f), \quad f \in H^2,$$ where $P_+$ denotes here the Riesz projection i.e. the orthogonal projection of $L^2(\T)$ onto $H^2$. It is a bounded operator on $H^2$ whose norm is equal to $\|\varphi\|_\infty$. See \cite[Section 4]{espMod} for the main properties about Toeplitz operators with symbols in $L^\infty(\T)$. From now on, assume  that $\varphi \in H^\infty$, and note   that $T_\varphi$ is then the multiplication operator by $\varphi$. We have $\ker(T_\varphi) = \{0\}$ and $\ker(T_\varphi^*) = \mathcal{K}_\theta$, thus $$\codim(\mbox{Im}(T_\varphi)) = \dim(\ker (T_\varphi^*)) = \dim (\mathcal{K}_\theta)$$ where $\mathcal{K}_\theta := (\theta H^2)^\perp$ is the \textit{model space} associated with $\theta$ the inner part of $\varphi$. We refer the reader to \cite{espMod} for a very nice introduction to  model space theory. We also recall  that $T_\varphi$ is an isometry on $H^2$ if and only if $\varphi$ is inner. 

\begin{theo}\label{plongToepInt}
    Let $\varphi$ be a non constant inner function. Then $T_\varphi$ is embeddable into a $C_0$-semigroup on $H^2$ if and only if $\varphi$ is not a finite Blaschke product. Moreover, the operators of the semigroup are analytic Toeplitz operators if and only if $\varphi$ does not have any zero in $\D$. 
\end{theo}

\begin{proof}
    Let $\varphi$ be a non constant inner function. Then $T_\varphi$ is an isometry on $H^2$. By Theorem \ref{PlongIsoTanja}, $T_\varphi$ is embeddable if and only if $\codim(\mbox{Im} (T_\varphi)) = \dim(\mathcal{K}_\varphi) = \infty$. We deduce then easily that $T_\varphi$ is embeddable into a $C_0$-semigroup on $H^2$ if and only if $\varphi$ is not a finite Blaschke product. See \cite[Proposition 5.5.19]{espMod}. Denote by $(R_t)_{t \ge 0}$ this semigroup. Let us recall that the commutant $\{S\}'$ of $S$ on $H^2$ is gived by $$\{S\}' = \left \{ T_\psi : \psi \in H^\infty \right \}.$$ In that case, $(R_t)_{t \ge 0}$ is a semigroup of analytic Toeplitz operators if and only there exists $C \in \mbox{Hol}(\D)$ satisfying $\sup \{ \mbox{Re} (C(z)) : z \in \D \} < \infty$ and such that $R_t = T_{e^{tC}}$ for every $t \ge 0$ \cite{SGToep}. In particular, we get $R_1 = T_\varphi = T_{e^C}$ and $\varphi = e^C$, which does not vanish on $\D$. Reciprocally, if $\varphi$ does not vanish on $\D$, then $\varphi$ is an inner singular function of the form $$\varphi(z) = \exp \left \{- \int_{\T} \frac{\zeta + z}{\zeta - z} ~d\mu(\zeta) \right \}, \quad z \in \D$$
    for $\mu$ a finite positive measure of Borel on $\T$ such that $\mu$ is singular with respect to the Lebesgue measure $m$. It is easy to see, by considering, for every $t \ge 0$, the bounded analytic functions $$\varphi_t(z) = \exp \left \{- t\int_{\T} \frac{\zeta + z}{\zeta - z} ~d\mu(\zeta) \right \}, \quad z \in \D,$$ that $T_\varphi$ is embeddable into the $C_0$-semigroup $(M_{\varphi_t} = T_{\varphi_t})_{t \ge 0}$ of analytic Toeplitz operators on $H^2$.
\end{proof}
The case of the isometric Toeplitz operators is now complete. The aim of the rest of this section is to investigate other analytic  Toeplitz operators.   
 \begin{lem}\label{plongToepExt}
    Let $\varphi$ be an outer function. Then $T_\varphi$ is embeddable into a $C_0$-semigroup of analytic Toeplitz operators on $H^2$.
\end{lem}

\begin{proof}
    It is an immediate consequence of the canonical representation of an outer function. 
\end{proof}

\begin{prop}\label{plongToepSemiGeneral}
    Let $\varphi = (BS_\mu) \varphi_e \in H^\infty$ where $B$ is a Blaschke product, $S_\mu$ is an inner singular function and $\varphi_e$ is an outer function. Assume that $\varphi$ is not an outer nor an inner function. Then the following assertions hold:
        \begin{itemize}
            \item[(i)] if $B \equiv 1$, $T_\varphi$ is embeddable into a $C_0$-semigroup of analytic Toeplitz operators on $H^2$.
            \item[(ii)] if $S_\mu \equiv 1$ and if $B$ is a non constant finite Blaschke product, $T_\varphi$ is not embeddable into a $C_0$-semigroup on $H^2$.
        \end{itemize}
\end{prop}

\begin{proof}
    We have  
        \begin{itemize}
            \item[(i)] if $B \equiv 1$, then $\varphi = S_\mu \varphi_e$ which does not vanish on $\D$. Let us remark that $T_\varphi = T_{S_\mu \varphi_e} = M_{S_\mu}M_{\varphi_e}$. Each term is embeddable into a $C_0$-semigroup on $H^2$ which commutes, then $T_\varphi$ is embeddable into the product of these two semigroups.
            \item[(ii)] if $S_\mu \equiv 1$ and $B$ is a non constant finite Blaschke product, then we get $$\codim(\mbox{Im}(T_\varphi)) = \dim (\mathcal{K}_B) \notin \{0, \infty \}.$$ We conclude with \cite[Theorem V.1.7]{TanEis}. \qedhere
        \end{itemize}
\end{proof}


	
    Let us  remark that according to Proposition \ref{plongToepSemiGeneral}, the remaining open  question is the following.  
    \begin{quest}
    Do we have the embedding of $T_\varphi$ when $\varphi = B\phi$, with $B$ a non constant   Blaschke product and $\phi$ a non vanishing  analytic function on $\D$?
    \end{quest}
     For that purpose, let us just note that, on one side  $T_B$ is  embeddable if and only if $B$ is an infinite Blaschke product by  Theorem \ref{plongToepInt}.  On the other side $T_{\phi}$ is embeddable from Proposition \ref{plongToepSemiGeneral} $(i)$.
   The following examples show the difficulty and the interest of this open question.
    \begin{itemize}
        \item Let $B_1$ be a finite Blaschke product and $B_2$ be an infinite Blaschke product. Then $T_{B_1B_2}$ is embeddable by Theorem \ref{plongToepInt}, whereas $T_{B_1}$ is not embeddable.  
        \item Let $B$ be an infinite Blaschke product and $S_\mu$ be a singular inner function. Then $T_{B}$ and $T_{S_\mu}$ are embeddable into a $C_0$-semigroup denoted respectively by $(A_t)_{t \ge 0}$ and $(B_t)_{t \ge 0}$ by Theorem \ref{plongToepInt}. Moreover $T_{BS_\mu}$ is also embeddable into a $C_0$-semigroup, which is not the product of $(A_t)_{t \ge 0}$ and $(B_t)_{t \ge 0}$, even though $T_B$ and $T_{S_\mu}$ commute.  
    \end{itemize}

We end  this section with  a result on the embedding of  Toeplitz operators whose symbol are   polynomials. 
\begin{coro}\label{plongPolyToep}
    Let $n \ge 1$ and $P \in \mathcal{P}_n$, where $\mathcal{P}_n$ is the space of polynomials of degree at most $n$. Then $T_P$ is embeddable into a $C_0$-semigroup on $H^2$ if and only if $P$ does not have any zero in $\D$. 
\end{coro}

\begin{proof}
    Let $n \ge 1$ and $P \in \mathcal{P}_n$ of the form $$P(z) = a\prod_{k = 1}^n (z - \alpha_k)\prod_{j = 1}^m (z - \beta_j)$$ where $a \in \C \backslash \{0\}$, $\abs{\alpha_k} < 1$ for every $1 \le k \le n$ and $\abs{\beta_j} \ge 1$ for every $1 \le j \le m$. It follows that $P(z) = B(z)F(z)$ where $B$ is the finite Blaschke product associated with the sequence $(\alpha_k)_{1 \le k \le n}$ and $F$ is the outer function defined by $F(z) = a \prod_{k = 1}^n (1 - \overline{\alpha_k}z) \prod_{j = 1}^m (z - \beta_j)$. Then 
\begin{itemize}[label = \textbullet]
    \item if $P$ does not have any zero in $\D$ i.e. $\alpha_k \notin \D$ for every $1 \le k \le n$, then $B \equiv 1$ and $P$ is outer. Therefore, $T_P$ is embeddable into a $C_0$-semigroup on $H^2$ according to Lemma \ref{plongToepExt}.
    \item if $P$ has at least one zero in $\D$, then $B \not\equiv 1$. By Proposition \ref{plongToepSemiGeneral} $(ii)$, $T_P$ is not embeddable into a $C_0$-semigroup on $H^2$. \qedhere
\end{itemize}
\end{proof}

\section{Isometric operators and properties of semigroups}\label{PropOpIso}

In this last section we state two quite obvious results concerning  the properties of the semigroup in terms of isometry or compactness where the operator embedded is isometric. 

\begin{prop}
    Let $V \in \Lcont(H)$ be isometric. If $V$ is embeddable into a $C_0$-semigroup of contractions $(V_t)_{t \ge 0}$ on $H$, then $V_t$ is isometric for every $t \ge 0$.  
\end{prop}

\begin{proof}
    Let us remark that since $V$ is an isometry, $V^n = V_n$ is also an isometry for every $n \in \N$. Assume that there exists $t_0 > 0$ such that $V_{t_0}$ is not isometric. In that case, since $V_{t_0}$ is a contraction, there exists $x_0 \in H$, $\norm{x_0} = 1$ such that $\norm{V_{t_0}x_0} < 1$. But, for every $N > t_0$, we have on one hand $$\norm{V_{N - t_0}V_{t_0}x_0} = \norm{V_N x_0} = 1$$ and on the other hand $$\norm{V_{N - t_0}V_{t_0}x_0} \le \norm{V_{t_0}x_0} < 1.$$ We obtain a contradiction, and so $V_t$ is isometric for every $t > 0$.
\end{proof}

\begin{prop}
    Let $V \in \Lcont(H)$ be isometric. If $V$ is embeddable into a $C_0$-semigroup $(V_t)_{t \ge 0}$ on $H$, then $V_t$ is not compact for every $t \ge 0$. 
\end{prop}

\begin{proof}
    Assume that there exists $t_0 > 0$ such that $V_{t_0}$ is compact. Since $\mathcal{K}(H)$ is a bilateral ideal, it comes that $V_t$ is compact for every $t \ge t_0$ from the algebraic property of the semigroup. It comes also that, for every orthonormal sequence  $(e_n)_{n \ge 0}$ of $H$, $\norm{V_t e_n} \underset{n \rightarrow \infty}{\longrightarrow} 0$ for every $t \ge t_0$. But, since $V$ is isometric, $V_N = V^N$ is also isometric for every $N \in \N$ and we get $$\norm{V_N e_n} = \norm{e_n} = 1.$$ For $N \ge t_0$, we obtain a contradiction and so $V_t$ is not compact for every $t > 0$.  
\end{proof}

\textbf{Acknowledgments:} The authors are grateful to E. Fricain for his comments and useful remarks.

\bibliographystyle{plain}



\end{document}